\documentclass[12pt,reqno]{article}

\usepackage[usenames]{color}
\usepackage{amssymb}
\usepackage{graphicx}
\usepackage{amscd}

\usepackage[colorlinks=true,
linkcolor=webgreen,
filecolor=webbrown,
citecolor=webgreen,
pdfstartpage=1]{hyperref}

\definecolor{webgreen}{rgb}{0,.5,0}
\definecolor{webbrown}{rgb}{.6,0,0}

\usepackage{color}
\usepackage{float}

\usepackage{graphics,amsmath,amssymb}
\usepackage{amsthm}
\usepackage{amsfonts}
\usepackage{latexsym}

\newcommand{\seqnum}[1]{\href{http://oeis.org/#1}{\underline{#1}}}

\begin{document}

\theoremstyle{plain}
\newtheorem{theorem}{Theorem}
\newtheorem{corollary}[theorem]{Corollary}
\newtheorem{lemma}[theorem]{Lemma}
\newtheorem{proposition}[theorem]{Proposition}
\newtheorem{obs}[theorem]{Observation}

\theoremstyle{definition}
\newtheorem{definition}[theorem]{Definition}
\newtheorem{example}[theorem]{Example}
\newtheorem{conjecture}[theorem]{Conjecture}

\theoremstyle{definition}
\newtheorem{remark}[theorem]{Remark}
\begin{center}
\vskip 1cm

{\LARGE\bf Arc coloring of odd graphs for hamiltonicity}
\vskip 1cm
\large
Italo J. Dejter\\
University of Puerto Rico\\
Rio Piedras, PR 00936-8377\\
\href{mailto:italo.dejter@gmail.com}{\tt italo.dejter@gmail.com} \\
\end{center}

\begin{abstract}\noindent
Coloring the arcs of biregular graphs was introduced with possible applications to industrial chemistry, molecular biology, cellular neuroscience, etc. Here, we deal with arc coloring in some non-bipartite graphs. In fact, for $1<k\in\mathbb{Z}$, we find that the odd graph $O_k$ has an arc factorization with colors $0,1,\ldots,k$ such that the sum of colors of the two arcs of each edge equals $k$. This is applied to analyzing the influence of such arc factorizations in recently constructed uniform 2-factors in $O_k$ and in Hamilton cycles in $O_k$ as well as in its double covering graph known as the middle-levels graph $M_k$.
\end{abstract}

\section{Introduction}\label{s1}

Let $0<k\in\mathbb{Z}$, let $n=2k+1$ and let $O_k$ be the $k${\it -odd graph} \cite{B}, that we consider as the graph whose vertices are the $k$-subsets of the cyclic group $\mathbb{Z}_n$ over the set $[0,2k]=\{0,1,\ldots, 2k\}$ having an edge, denoted $uv$, between each two vertices $u,v$ if and only if $u\cap v=\emptyset$.

Coloring the arcs of biregular graphs was considered in \cite{DD}, with potential applications to the design of experiments for industrial chemistry, molecular biology, cellular neuroscience and solving 3-dimensional puzzles like the one known as Great Circle Challenge.
It would be also valuable to find likewise applications of arc coloring to graphs other than bipartite graphs, like the odd graphs, for example, and any other similar graphs in that all vertices have departing arcs with all weights (colors) from 0 to $k$ such that the sum of oppositely oriented arcs is constantly $k$.
In this work, coloring the arcs of $O_k$ occupies the place of missing 1-factorizations, since the Petersen graph $O_3$ is 4-edge-colorable and if $k$ is a power of 2 then $O_k$ is $k+1$-edge-colorable \cite{B}. Resulting arc-factorizations in Section~\ref{modular} are seen in Section~\ref{verL} to influence recent uniform 2-factors and Hamilton cycles of $O_k$ \cite{u2f,Hcs}.

In fact, we recur in Section~\ref{modular} to an {\it edge-supplementary 1-arc factorization} $\mathbb{A}_k$ of $O_k$, meaning that the two oppositely oriented arcs (1-arcs, in \cite[p.~59]{GR}) of each edge of $O_k$ are assigned {\it colors} $a,b\in\{0,\ldots,k\}=[0,k]$ by means of $\mathbb{A}_k$ such that $a+b=k$ (so $a,b$ are said to be $k$-{\it supplementary} or {\it supplementary in} $k$), in such a way that the arcs departing from each vertex are in one-to-one correspondence with $[0,k]$.

 To define the claimed edge-supplementary 1-arc factorization $\mathbb{A}_k$ of $O_k$, we consider in Section~\ref{modular} a partition of $V(O_k)$  into $\mathbb{Z}_n$-{\it classes}, namely the cyclic equivalence classes mod $n$.

To get these $\mathbb{Z}_n$-classes, we take each vertex $u$ of $O_k$ expressed as the characteristic vector of the subset $u\subset\mathbb{Z}_n$ it represents.
Each such vector is a binary string, or {\it bitstring}, namely a sequence of digits 0 and 1 said to be 0-{\it bits} and 1-{\it bits}, respectively.

The number of bits (resp., 1-bits) of a bitstring $u$ is said  to be its {\it length} (resp., its {\it weight}). Each $u\in V(O_k)$ can be seen as a bitstring of length $n$ said to be an {\it $n$-bitstring}.

\begin{example}\label{o1} In $O_1$, the subsets $\{i\}$  of $\mathbb{Z}_3$, ($i=0,1,2$) are denoted $100,010,001$, respectively. \end{example}

We also consider the vertices of $O_k$ as corresponding  polynomials mod $x^n+1$ in the ring $\mathbb{Z}[x]$,
\cite{D2,D1}, namely in Example~\ref{o1}: $x^0$, $x^1$ and $x^2$ mod $x^3 +1$.

The $\mathbb{Z}_n$-classes of $O_k$ are obtained by successive multiplication of such polynomials by $x$ mod $x^n+1$. The resulting equivalence relation defines a quotient graph of $O_k$ whose vertices are those $\mathbb{Z}_n$-classes.
In Example~\ref{o1}, $O_1$ has just one such equivalence class, and $O_2$  has two.

Theorem~\ref{bij} below asserts that there is a bijection between the $\mathbb{Z}_n$-classes of $O_k$ and the {\it Dyck words} of length $2k$, defined in Subsection~\ref{Dyck}  via Example~\ref{PLC} (to Subsection~\ref{alfalfa}), namely with the roles of 0- and 1-bits exchanged with respect to the Dyck words of \cite{Hcs}. This
allows to determine $\mathbb{A}_k$ in Section~\ref{modular} and an arc-coloring analysis through $\mathbb{A}_k$ (not covered in \cite{Hcs}) of:
\begin{enumerate}
\item[\bf(i)] the uniform 2-factors of $O_k$ \cite{u2f} (as in Theorem~\ref{L5}, via Subsections~\ref{rev}-\ref{uu}) and the Hamilton cycles \cite{Hcs} of $O_k$, for $k>2$  (as in Theorem~\ref{L6}, via Subsections~\ref{joder}-\ref{r4});
\item[\bf(ii)] the double covering graph $M_k$ of $O_k$, namely the {\it middle-levels graph} of the {\it Boolean lattice} $B_n$ induced by the levels $L_k$ and $L_{k+1}$ of $B_n$, formed by the $n$-bitstrings of weight $k$ and $k+1$ (with Hamilton cycles lifted from those in item (i), see Corollary~\ref{the-end});
\item[\bf(iii)] the explicit {\it modular} 1-factorization of the graphs $M_k$ \cite{DKS}, with factor colors in $[1,k+1]$ obtained from the color set $[0,k]$ in Section~\ref{modular} by uniformly adding 1.
\end{enumerate}
The modular 1-factorization of $M_k$  cited in item (iii) differ from the {\it lexical} 1-factorization of $M_k$ \cite{KT}. For example, there are at least two different approaches to Hamilton cycles in $M_k$, namely: via the modular 1-factorization \cite{u2f} for $O_k$ in \cite{Hcs}, these represented below in Corollary~\ref{the-end},
as well as via the lexical 1-factorization for $M_k$ (never $O_k$) in \cite{gmn,M}.

\section{Restricted growth strings and k-germs}\label{germs}

To unify presentation of the odd graphs $O_k$, let us consider the sequence $S_{(\infty)}$ \cite[\seqnum{A239903}]{oeis} of {\it restricted-growth strings} ({\it \!RGS}) \cite[p.~325]{Arndt} and the $k$-th Catalan number $C_k=\frac{(2k)!}{k!(k+1)!}$ \cite[\seqnum{A000108}]{oeis}. The first $C_k$ terms of $S_{(\infty)}$ form a set $S_{(u)}$
\cite[p.~222]{Stanley} equivalent to the set $S_{(i)}$ of Dick paths from $(0,0)$ to $(2k,0)$ \cite[p.~221]{Stanley}. Both $S_{(u)}$ and $S_{(i)}$ are items in \cite[ex.~6.19]{Stanley}.

 The sequence ${S_{(\infty)}}$ is expressible as ${S_{(\infty)}}=(\beta(0),\beta(1),\beta(2),\ldots,\beta(17),\ldots)=$
$$(0,1,10,11,12,100,101,110,111,112,120,121,122,123,1000,1001,1010,1011,\ldots),$$
with the lengths of any two contiguous terms $\beta(m-1)$ and $\beta(m)$ ($1\le m\in\mathbb{Z}$) constant unless $m=C_k$, for some integer $k>1$, in which case $\beta(m-1)=\beta(C_k-1)=12\cdots k$ has length $k$ and $\beta(m)=\beta(C_k)=10^k=10\cdots 0$ has length $k+1$.

To manipulate $O_k$ and $M_k$ ($k>1$) in relation to their Hamilton cycles \cite{gmn,M,u2f}, we {\it dress} the RGS's $\beta=\beta(m)$ with length$(\beta)\le k$ as strings of fixed length $k-1$ that we call {\it $k$-germs}  \cite[p.~138]{D2} \cite[p.~8]{D1} in order to show (via the {\it nested castling} of Theorem~\ref{thm1} and  Subsection~\ref{alfalfa}) that $k$-germs form the domain of a bijection $f$ onto the Dyck words of length $2k$. Concretely, we make the {\it $k$-germ} of any such RGS $\beta=\beta(m)$ to be the $(k-1)$-string $\alpha=\alpha(m)=a_{k-1}a_{k-2}\cdots a_2a_1$ obtained from $\beta$ by prefixing $k-$length$(\beta)$ zeros to it.
This makes a $k$-{\it germ} to be a $(k-1)$-string $\alpha=a_{k-1}a_{k-2}\cdots a_2a_1$ such that:

\begin{enumerate}
\item[{\bf(1)}] the leftmost position of $\alpha$, namely position $k-1$, has entry $a_{k-1}\in\{0,1\}$;
\item[{\bf(2)}] given $1<i<k$, the entry $a_{i-1}$ at position $i-1$ satisfies $0\le a_{i-1}\le a_i+1$.
\end{enumerate}

\noindent Note that every $k$-germ $a_{k-1}a_{k-2}\cdots a_2a_1$ yields a $(k+1)$-germ $0a_{k-1}a_{k-2}\cdots$ $a_2a_1$.

To {\it undress} a $k$-germ $\alpha=\alpha(m)=a_{k-1}a_{k-2}\cdots a_1$ $\ne 00\cdots 0$ means that a {\it non-null RGS} is obtained by stripping $\alpha$ of its null entries to the left of its leftmost entry equal  to 1, in which case we denote such a non-null RGS also by $\alpha=\alpha(m)$.
To complement this notion, we say that the {\it null RGS} $\alpha=\alpha(0)=0$ corresponds to all null $k$-germs $\alpha=\alpha(0)$, for $0<k\in\mathbb{Z}$.

We consider also the {\it empty RGS}, denoted $\alpha=\phi$, that yields for $k=1$ the only {\it empty} $k$-germ $\alpha=0^{k-1}=0^{1-1}=\phi$, using the same notation $\phi$ both for the empty RGS and the empty 1-germ and extending this way the general notation $\alpha=0^{k-1}$ ($k>1$) to every $k>0$.

There are exactly $C_k$ $k$-germs $\alpha=\alpha(m)<10^k$, $\forall k>0$.
Given two $k$-germs
$\alpha=a_{k-1}\cdots a_2a_1$ and $\beta=b_{k-1}\cdots b_2b_1,$ ($\alpha\ne \beta$), $\alpha$ is said to be less than $\beta$, written $\alpha<\beta$, if

\begin{enumerate}
\item[{\bf (i)}] either $0=a_{k-1} < b_{k-1}=1$
\item[{\bf (ii)}] or $\exists i\in[2,k-1]$ such that $a_{k-i} < b_{k-i}$ with $a_{k-j}=b_{k-j}$, $\forall j\in[1,i-1]$.
\end{enumerate}

The resulting order on $k$-germs $\alpha(m)$ corresponds bijectively with the natural order of the integers $m\in[0,C_k]$, via the assignment $m\rightarrow\alpha(m)$.

\section{Ordered trees of k-germs and Dyck words}\label{nat}

We recall from \cite[Theorem 3.1]{D2} or \cite[Theorem 1]{D1} that the $k$-germs are the nodes of an ordered tree ${\mathcal T}_k$ rooted at $0^{k-1}$ and such that each $k$-germ $\alpha=a_{k-1}\cdots a_2a_1\ne0^{k-1}$ with rightmost nonzero entry $a_i$  ($1\le i=i(\alpha)<k$) has parent $\beta(\alpha)=b_{k-1}\cdots b_2b_1\!<\alpha$  in ${\mathcal T}_k$ with $b_i= a_i-1$ and $a_j=b_j$, for every $j\ne i$ in $[1,k-1]$.

\begin{example}\label{ej1}
By representing ${\mathcal T}_k$ with each node $\beta$ having its children $\alpha$ enclosed between parentheses following $\beta$ and separating siblings with commas, we can write: $${\mathcal T}_4=000(001,010(011(012)),100(101,110(111(121)),120(121(122(123))))).$$\end{example}

\begin{theorem}\label{thm1} {\bf (i-nested castling)}
To each $k$-germ $\alpha=a_{k-1}\cdots a_1$
corresponds an $n$-string $F(\alpha)$ whose entries are the numbers $0,1,\ldots,k$, once each, and $k$ ``$=$" signs. Moreover,
$$F(0^{k-1})=``012\cdots(k-2)(k-1)k=\cdots =",\; (e.g,\; F(0)=``01=",\; F(00)=``012==").$$
Furthermore, if $\alpha\ne 0^{k-1}$, let
\begin{enumerate}\item $W^i$ and $Z^i$ be the leftmost and rightmost, respectively, substrings of length $i=i(\alpha)$ in $F(\beta)$, where $\beta$ is the parent of $\alpha$ in $\mathcal T_k$;
\item $c>0$ be the leftmost entry of $F(\beta)\setminus(W^i\cup Z^i)$, and
\item $F(\beta)\setminus(W^i\cup Z^i)$ be the concatenation $X|Y$, where $Y$ starts at the entry $c+1$ of $F(\beta)$.
\end{enumerate}
Then $F(\alpha)=W^i|Y|X|Z^i$ is the $i$-{\rm nested castling} of $F(\beta)=W^i|X|Y|Z^i$.
In addition, if an entry $b'\in[0,k]$ of $F(\alpha)$ is followed immediately to its right by an entry $b\in[0,k]$, then $k\ne b'<b$.
 Also, $W^i$ is an ascending number $i$-substring, $Z^i$ is formed by $i$ signs``$=$", and $``k="$  is a substring of $F(\alpha)$, but $``=k"$  is not.
\end{theorem}

\begin{proof}  It was proved in \cite[Theorem 3.2]{D2}, as well as in \cite[Theorem 2]{D1}, where asterisks, ``$*$'', were used instead of the present ``$=$'' signs. Examples~\ref{ej2} and~\ref{PLC} yield ideas on the proof.
\end{proof}

\begin{figure}[htp]
\includegraphics[scale=0.66]{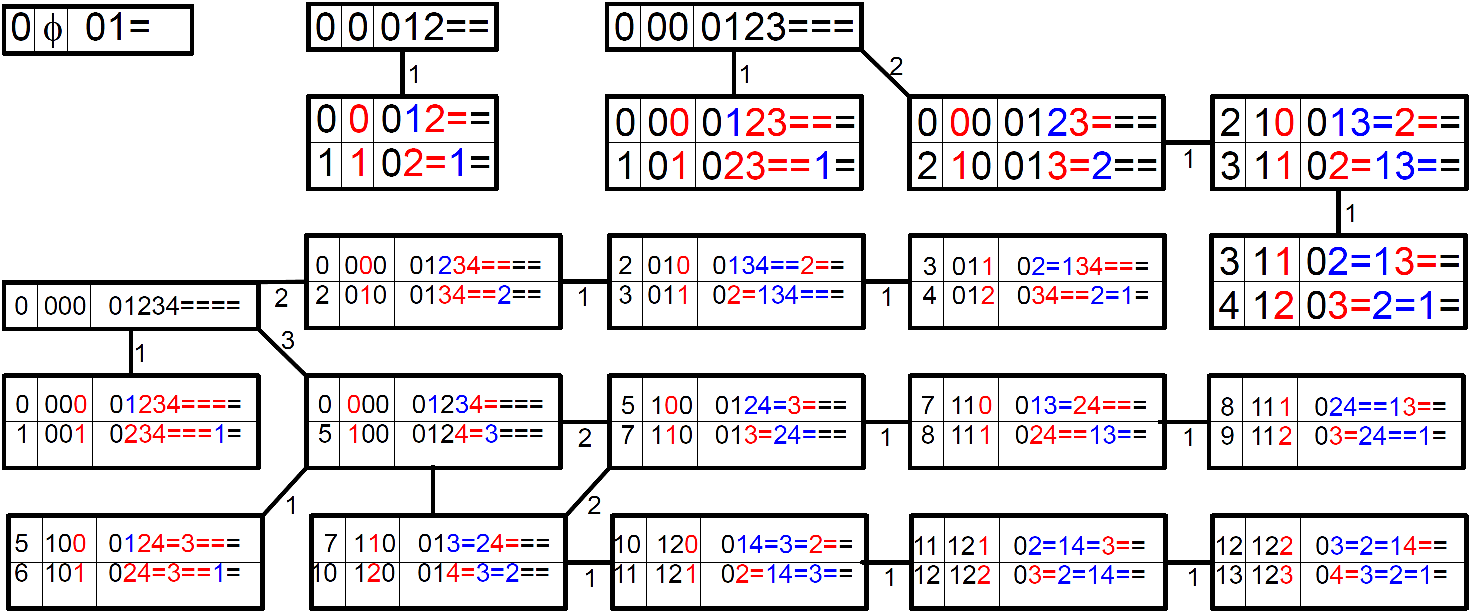}
\caption{Exemplifying Theorem~\ref{thm1} for the ordered trees ${\mathcal T}_k$ ($k=1,2,3,4)$ }
\label{fig1}
\end{figure}

\begin{example}\label{ej2}
Fig.~\ref{fig1} shows each tree ${\mathcal T}_k$ ($k=1,2,3,4$), with its root $0^{k-1}$ represented in a box containing the order $ord(0^{k-1})=0$, the root $0^{k-1}$ and                                                                      $F(0^{k-1})$. Each other node $\alpha$ of ${\mathcal T}_k$ is represented by a box of two levels: the top level contains the order $ord(\beta(\alpha))$, the parent $\beta(\alpha)$ and $F(\beta(\alpha))$; the lower level
contains the order $ord(\alpha)$, $\alpha$ and $F(\alpha)$. In these presentations of $\beta(\alpha)$ and $\alpha$, the entries $b_i$ and $a_i$ are colored red and the remaining entries black. In all boxes, $F(\beta(\alpha))=``W^i|X|Y|Z^i"$ and $F(\alpha)=``W^i|Y|X|Z^i"$ have $X$ and $Y$ colored blue and red, respectively, while $W^i$ and $Z^i$ are left black. In addition, the edge leading from $\beta(\alpha)$ to $\alpha$ is labeled with its subindex $i$.
\end{example}

\subsection{Bitstring forms out of nested castling}\label{alfalfa}

For each $k$-germ $\alpha$ ($k>1$), let us define the bitstring form $f(\alpha)$ of $F(\alpha)$ by replacing each number entry of $F(\alpha)$ by a 0-bit and each ``$=$" sign by a 1-bit. (0-bits and 1-bits here correspond respectively to the 1-bits and 0-bits of \cite{Hcs}). Such $f(\alpha)$ is an $n$-bitstring of weight $k$ whose support $supp(f(\alpha))$ is in $V(O_k)$. So, we consider both
$F(\alpha)$ and the characteristic vector $f(\alpha)$ of $supp(f(\alpha))$ to represent the vertex $supp(f(\alpha))$ of $O_k$.

\begin{example}\label{PLC}
We can recover $F(\alpha)$ from $f(\alpha)$, exemplified for $k=1,2,3$ in Fig.~\ref{fig2}. In it, for each one of the $1+2+5=8$ cases in the figure,
a piecewise-linear curve $PLC(\alpha)$ is constructed iteratively that starts at the shown origin O in the Cartesian plane $\Pi$ by
replacing successively the 0-bits and 1-bits of $f(\alpha)$ by {\it up-steps} and {\it down-steps}, namely diagonal segments $(x,y)(x+1,y+1)$ and $(x,y)(x+1,y-1)$, respectively. To each down-step of $PLC(\alpha)$, we assign the ``$=$"- sign.
 We assign the integers in the interval $[0,k]$ in decreasing order (from $k$ to 0) to the up-steps of $PLC(\alpha)$, from the top unit layer of $PLC(\alpha)$ in $\Pi$ to the bottom one and from left to right at each pertaining unit layer between contiguous lines $y,y+1\in\mathbb{Z}$.
Then, by reading and successively writing the number entries and ``$=$" signs assigned to the steps of $PLC(\alpha)$, the $n$-tuple $F(\alpha)$ is obtained.
Fig.~\ref{fig2} is provided, underneath each instance, with the corresponding $k$-germ $\alpha$ followed by  $F(\alpha)$ and its (underlined) order of presentation via Theorem~\ref{thm1}. We assume that all elements of $V(O_k)$ are represented by means of such piecewise-linear curves, for each fixed integer $k>0$.

Theorem~\ref{thm1} is exemplified in Fig.~\ref{fig2} too, where {\it $i$-nested castling} is occurring via {\it layer polygons} (either isosceles trapezoids or triangles) with their interiors pairwise disjoint, as follows. For $k=2$: between the  blue layer polygon (delimited by the up-step ``1'' on the left and the down-step ``='' on the right)  and the yellow layer polygon (delimited by the up-step ``2'' on the left and the down-step ``='' on the right). For $k=3$: between the blue, green and yellow layer polygons (delimited on the left by the up-steps ``1'', ``2'' and ``3'', respectively, and corresponding down-steps ``='' on their right).

Specifically in Fig~\ref{fig2}: For $k=2$, the 1-nested castling from the root 2-RGS (0) to the 2-RGS (1) is depicted as a permutation of the contiguously labelled up-steps of the (possibly shortened) layer polygons. For $k=3$, the 1- and 2-nested castlings from the root RGS (00) to the RGS's (01) and (10), respectively, permute the order of the contiguously labelled up-steps of the (possibly shortened) layer polygons as indicated in the figure. Similarly for the 1-nested castlings from (10) to (11) and from (11) to (12).
\end{example}

\subsection{Dyck paths}\label{Dyck}

\begin{figure}[htp]
\includegraphics[scale=0.73]{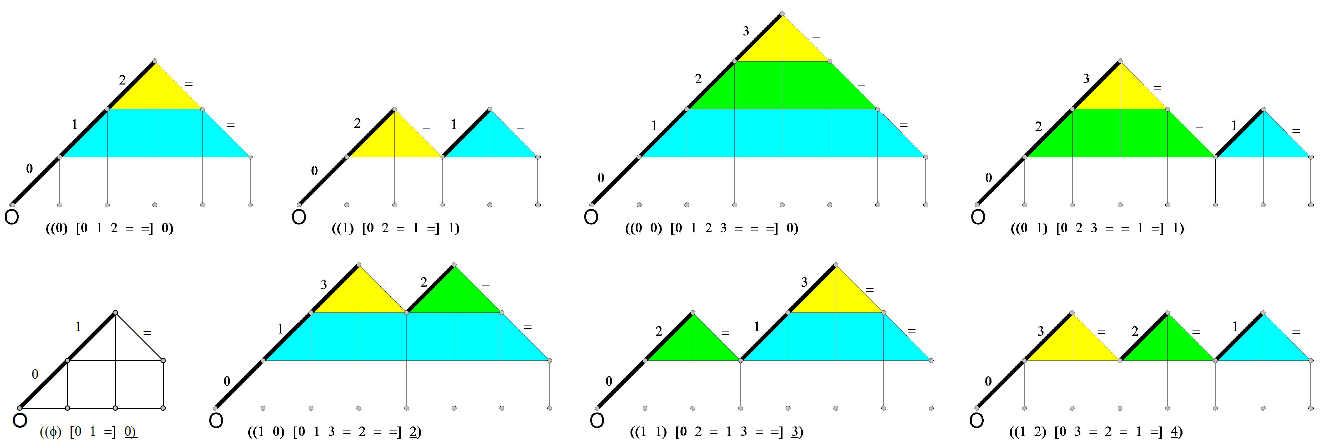}
\caption{Recovering $F(\alpha)$ from $f(\alpha)$: $PLC(\alpha)$ for triples (($\alpha)$ $[F(\alpha)]$, $\underline{ord(\alpha)})$, $k=1,2,3$}
\label{fig2}
\end{figure}

Let $0<k\in\mathbb{Z}$ and let $\alpha$ be a $k$-germ. The curve $PLC(\alpha)$ (Example~\ref{PLC} and Fig.~\ref{fig2}) yields a {\it Dyck path} $DP(\alpha)$ via the removal of its first up-step $(0,0)(1,1)$ and a change of coordinates from $(1,1)$ to $(0,0)$. Such Dyck path $DP(\alpha)$ represents a corresponding {\it Dyck word} $DW(\alpha)=``0\cdots 1"$ of length $2k$, a particular case for $\ell=k$ of a {\it Dick word of length} $2\ell$ ($0<\ell\in\mathbb{Z}$), defined as a $2\ell$-bitstring
of weight $\ell$ such that in every prefix the number of 0-bits is at least the number of 1-bits
 (differing from the Dyck words of \cite{Hcs}, in which, on the contrary, the number of 1-bits is at least equal to the number of 0-bits).
 The concept of {\it empty Dyck word} $\epsilon$ also makes sense here and is used for example in Section~\ref{verL}, display~(\ref{!}).
The Dyck paths $DP(\alpha)$ corresponding to the curves $PLC(\alpha)$ in Fig.~\ref{fig2} are represented in the lower-left quarter of Fig.~\ref{fig3}, with notation specified in Examples~\ref{rr2} and \ref{rrr2}, and preserving the colors of Fig.~\ref{fig2}. In Subsection~\ref{uu}, the down-steps on the right of the layer polygons of Fig~\ref{fig2} will have their labels ``='' changed to ``$\underline{j}$'', if ``$j$'' is the label of the associated up-step. This takes $F(\alpha)$ into an $n$-string $\underline{F}(\alpha)$, whose substrings $[j,\underline{j}]$ project in $f(\alpha)$ as Dyck subwords.

\begin{theorem}\label{bij}
There exists a bijection $\lambda$ from the $\mathbb{Z}_n$-classes of $V(O_k)$ onto the Dyck words of length $2k$.
In fact, each $\mathbb{Z}_n$-class $\Gamma$ of $V(O_k)$ has a Dyck word $f(\alpha)$ of length $2k$ as sole representative. The other $n$-tuples in $\Gamma$ are obtained by translations $f(\alpha).j$ mod $n$ of $f(\alpha)$, where $j\in[0.2k]$ is the position of the null entry in $f(\alpha).j$.
Also, $f(\alpha)$ may be interpreted as its corresponding $F(\alpha)$ and the other $n$-tuples $f(\alpha).j$ above may be interpreted as the corresponding translations $F(\alpha).j$ mod $n$.
\end{theorem}

\begin{proof} The $n$-tuples $F(\alpha)$ were obtained via the $i$-nested castlings of Theorem~\ref{thm1} associated to the indices $i=i(\alpha)$ of the oriented edges $\beta\alpha$ of the tree ${\mathcal T}_k$ of Section~\ref{nat} from the parent $\beta$ of each non-root $k$-germ $\alpha$ to $\alpha$.
Note that there are just $C_k$ Dick words of length $2k$ (Subsection~\ref{Dyck}) corresponding bijectively to the $n$-tuples $F(\alpha)$, and to their binary versions $f(\alpha)$ (Subsection~\ref{alfalfa}). Also, there are exactly $C_k$ $\mathbb{Z}_n$-classes $\Gamma$ in $V(O_k)$. Then,
each $\mathbb{Z}_n$-class $\Gamma$ of $O_k$ contains a sole $f(\alpha)$, which correspond to a sole $F(\alpha)$ by the approach in Example~\ref{PLC}. As a result,
the correspondence from the $\mathbb{Z}_n$-classes $\Gamma$ onto such Dyck words is a bijection $\lambda$ with $\lambda(\Gamma)=f(\alpha)$, as claimed, and we may write $\Gamma=\Gamma_\alpha=\lambda^{-1}(f(\alpha))$.

\begin{figure}[htp]
\includegraphics[scale=0.73]{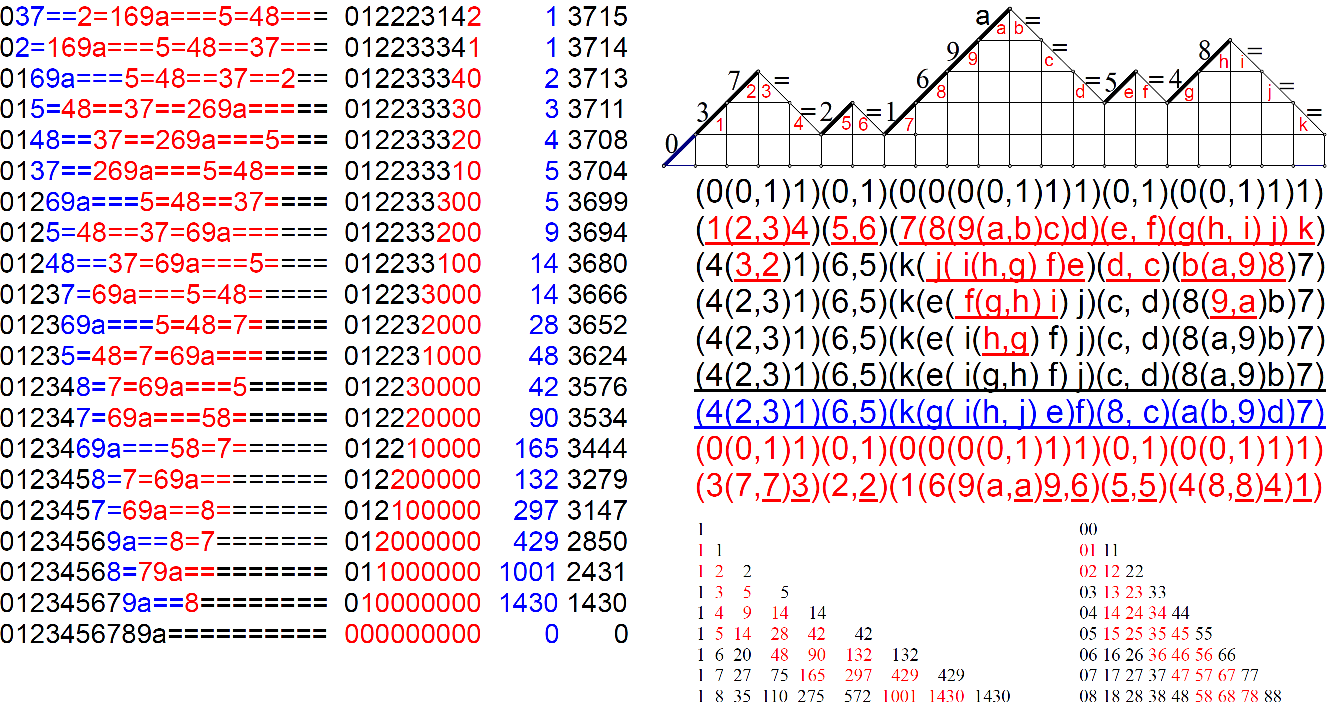}
\caption{Illustration for Examples~\ref{Toc} and~\ref{Inc} and Subsection~\ref{rev}}
\label{fig4}
\end{figure}

With respect to the last two sentences in the statement, note that $\mathbb{Z}_n$ acts on $O_k$ yielding, for each $k$-germ $\alpha$, the orbit of $f(\alpha)$, including the translations of $f(\alpha)$, namely:
$$f(\alpha)=f(\alpha).0=f_0f_1\cdots f_{2k},\hspace*{5mm}
f(\alpha).1=f_1f_2\cdots f_{2k}f_0,\hspace{5mm} f(\alpha).2=f_2f_3\cdots f_{2k}f_0f_1,\hspace{5mm}\ldots,$$ where $f_j=f_j(\alpha)\in\{0,1\}$ ($j\in[0,2k]$), with general term given by
$$f(\alpha).j=f_jf_{j+1}\cdots f_{2k}f_0f_1\cdots f_{j-1},$$ and ending up with $f(\alpha).2k=f_{2k}f_0f_1\cdots f_{2k-1}$.

Similar treatment holds by taking $F(\alpha)=F(\alpha).0=F_0F_1\cdots F_{2k}$,
where $F_j=``="$ if $f_j$ is a 1-bit and $F_j\in[0,k]$ if $f_j$ is a 0-bit (the value of $F_j$ provided by the said approach in Example~\ref{PLC}),
with general translation term given by $$F(\alpha).j=F_jF_{j+1}\cdots F_{2k}F_0F_1\cdots F_{j-1}.$$ This covers all the vertices of $\mathbb{Z}_n$-classes of $V(O_k)$, seen either from the $f(\alpha)$ point of view or from the $F(\alpha)$ point of view.
\end{proof}

\begin{example}\label{Toc} In this and subsequent examples, we express integers in their hexadecimal form (e.g., $a=10,b=11$, etc.). To clarify concepts, let us determine the $n$-germ $\alpha_1$ ($n=21$) corresponding to the bitstring $f(\alpha_1)=00110100001110100111$.
We proceed by determining $PLC(\alpha_1)$ (as indicated in Example~\ref{PLC}),  drawn in the upper-right of Fig.~\ref{fig4}, where the black hexadecimal number entries and ``$=$" signs form the $n$-string $F(\alpha_1)$, while the red symbols
are the first twenty positive hexadecimal numbers, (that appear in that order in the expression $h_0(\alpha)=h_0(\alpha_1)$
of Subsection~\ref{rev}, item 2).
To associate the $k$-germ $\alpha_1$ to the $n$-string $F(\alpha_1)$, we build a list $\mathbb{L}(\alpha_1)$ shown on the left of Fig.~\ref{fig4}.
The first lines of $\mathbb{L}(\alpha_1)$ contain data concerning the path $P(\alpha_1)$ from $\alpha_1$ to the root $0^{20}=\alpha_{21}$ in ${\mathcal T}_{21}$, namely: $F(\alpha_i)$, $\alpha_i$, $ord(\alpha_i)-ord(\alpha_{i+1})$ and $ord(\alpha_i)$, for $i=1,2,\ldots,20$. The first sublist $\mathbb{L}'(\alpha_1)$ in $\mathbb{L}(\alpha_1)$, composed successively by $F(\alpha_1),\ldots,F(\alpha_{21})$, shows each of the 21-strings $F(\alpha_j)$, ($j=1,\ldots,20$), as a concatenation $``W^{i_j}|X|Y|Z^{i_j}"$, where $i_j$ is the first index in $F(\alpha_j)=c_0c_1\cdots c_{20}$ such that $c_j>j$ with blue $X$, red $Y$, and black for both $W^{i_j}$ and $Z^{i_j}$, showing in the following line the 21-string
$F(\alpha_{j+1})=``W^{i_j}|Y|X|Z^{i_j}"$, just under $F(\alpha_j)$. To the right of $\mathbb{L}'(\alpha_1)$ and starting at the red $\alpha_{21}=0^{k-1}$ in line 21, we went up and built a sublist $\mathbb{L}''(\alpha_1)$ by reconstructing each $\alpha_j=a_{k-1}\cdots a_1$, setting in red the terminal substring $a_{i_j}\cdots a_1$ and in black the initial substring $a_{k-1}\cdots a_{{i_j}+1}$.
To the right of $\mathbb{L}''(\alpha_1)$, we constructed an accompanying blue sublist $\mathbb{L}'''(\alpha_1)$ formed by Catalan numbers taken as increments that determine the corresponding orders of the vertices in $P(\alpha_1)$. These orders, appearing in the final sublist $\mathbb{L}''''(\alpha_1)$, are obtained as the partial sums of Catalan numbers. This takes to $ord(\alpha_1)=3715<4862=|V({\mathcal T}_{21})|$. The blue sublist $\mathbb{L}'''(\alpha_1)$ arises from the red entries in the first nine lines of Catalan's triangle in the lower part of Fig.~\ref{fig4} with entries $\tau_i^j$ as in \cite[pp.~139--140]{D2} represented as pairs $ij$ to the right of the said nine lines,
($i\in[0,8]$, $j\in[1,8]$).
\end{example}

\section{Edge-supplementary arc factorizations}\label{modular}

Each arc $(u,v)$ of $O_k$ ($u,v\in V(O_k)$) is represented by translations
$F(\alpha_u).j_u$ and $F(\alpha_v).j_v$   mod $n$ of the $n$-strings $F(\alpha_u)$ and $F(\alpha_v)$.
Looking $u$ and $v$ upon as $u=F(\alpha_u).j_u$ and $v=F(\alpha_v).j_v$ and comparing, we see that apart from a specific number entry $i\in[0,2k]$ in both $u$ and $v$, all other number entries of one of them correspond to ``$=$" sign entries of the other one, and vice versa. Moreover, the entries $u_i$ of $u$ and $v_i$ of $v$ satisfy $u_i+v_i=k$, so they are said to be $k$-{\it supplementary}. Then, the {\it edge-supplementary 1-arc factorization} $\mathbb{A}_k$ of $O_k$ claimed in Section~\ref{s1} is given by the values of those entries $u_i$ and $v_i$ taken as colors of the arcs $(u,v)$ and $(v,u)$, respectively, for all pairs of adjacent vertices $u$ and $v$ of $O_k$.

\begin{figure}[htp]
\hspace*{0.5mm}
\includegraphics[scale=0.72]{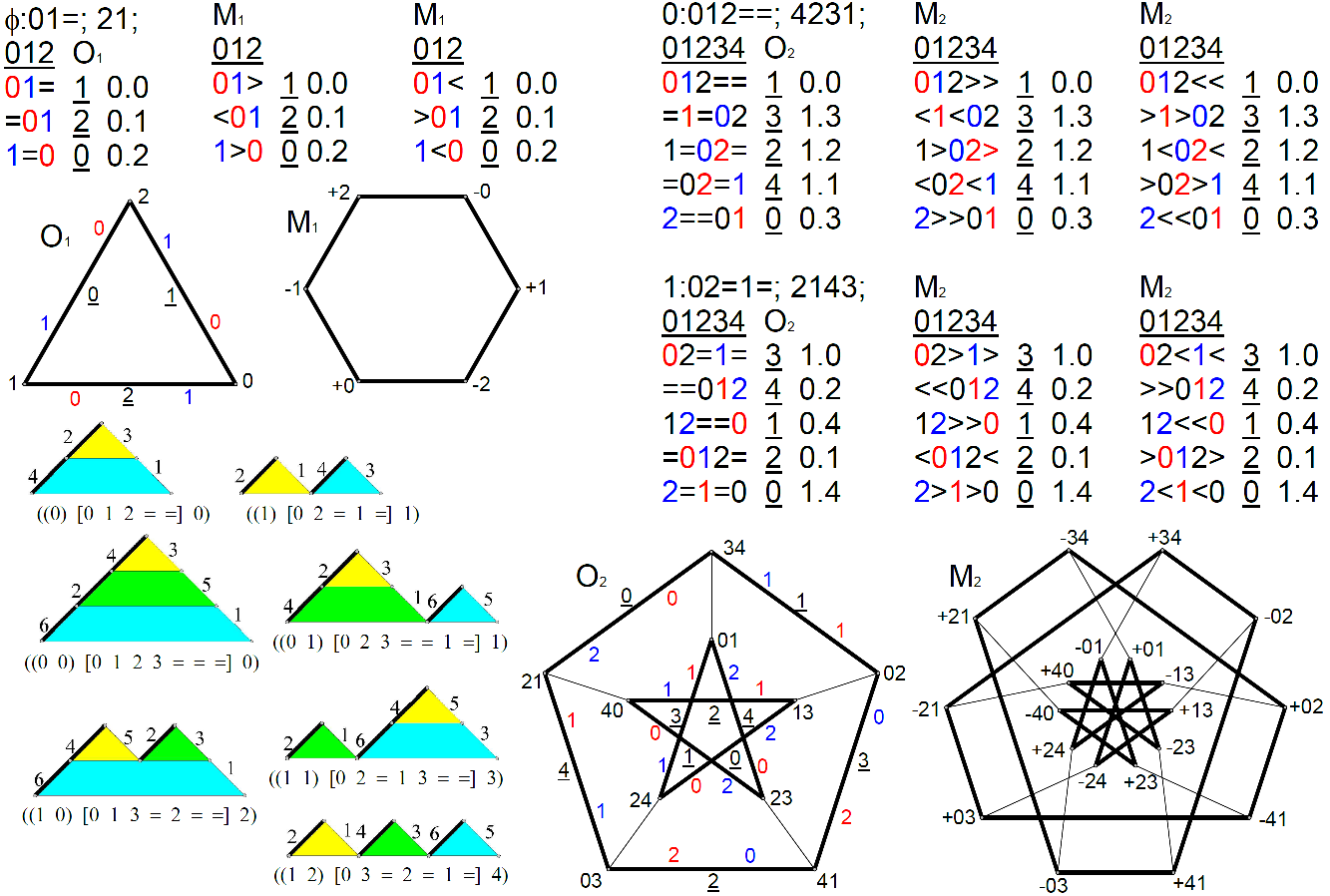}
\caption{Illustration for Section~\ref{modular} and Examples~\ref{rr2} and~\ref{rrr2}}
\label{fig3}
\end{figure}

\begin{example}\label{rr2}
Edge-supplementarity is illustrated for $k=1,2$ in Fig.~\ref{fig3}. In it, the $n$-tuples $F(\alpha)$ are shown as the initial lines of corresponding vertical lists $L(\alpha)$ in which arcs $(u,v)$ of $O_k$ appear as ordered pairs $(F(\alpha_u).j_u,F(\alpha_v).j_v)$ disposed on contiguous lines, except for all arcs from bottom lines, taken as $n$-tuples $F(\alpha_u).j_u$, to corresponding top lines, taken as associated $n$-tuples $F(\alpha_v).0$, thus closing the lists $L(\alpha)$ into oriented cycles $C(\alpha)$. In each such pair $(F(\alpha_u).j_u,F(\alpha_v).j_v)$, the $i$-entries $u_i$ and $v_i$ are colored respectively blue in $F(\alpha_u).j_u$ and red in $F(\alpha_v).j_v$ (the other entries in black) with the exception of the bottom and top $n$-tuples in each list: these are also adjacent, with the entry $u_i=u_0$ holding blue value $k$ on the bottom $n$-tuple $u$ and the entry $v_i=v_0$ holding red value 0 on the top $n$-tuple $v=F(\alpha_v)=F(\alpha_v)_0$. The position $i$ of the blue entry $u_i$ in each line of the lists is cited underlined ($``\underline{i}"$) to the right of its $n$-tuple $u$; the vertex $u\in O_k$ represented in such a line is still cited to the right of its $``\underline{i}"$ as $``ord(\alpha).j_u"$, where
$ord(\alpha)$ refers to the $\mathbb{Z}_n$-class $\Gamma_\alpha$ (so denoted in the proof of Theorem~\ref{bij}) of $u$ in $O_k$. Such vertical lists are used in
Section~\ref{verL} (Fig.~\ref{fig6},~\ref{fig5},~\ref{fig7}) in order to yield Hamilton cycles of $O_k$, for $k>2$, as in \cite{u2f,Hcs}; ($k=2$ is excluded; indeed, $O_2$ is the hypohamiltonian Petersen graph).
\end{example}

\subsection{String reversals in properly nested parentheses}\label{rev}

{\bf Assignment of a $2k$-permutation $\pi(\alpha)$ to each $k$-germ $\alpha$.}

\noindent Consider the Dyck path $DP(\alpha)$ obtained from $PLC(\alpha)$ by the removal of its first up-step and subsequent change of coordinates from $(1,1)$ to $(0,0)$.

\begin{enumerate}\item
Let $f'$ (resp., $F'$) be the $2k$-string obtained from $f$ (resp. $F$) by removing its first entry.
Set parentheses or commas between each two entries of $f'$, so that the four substrings
$$\begin{array}{rrrrrrrrl}
 ``01"&,&``10"&,&``00"&\mbox{and}&``11"&&\mbox{ are transformed into the substrings}\\
 ``0,1"&,&``1)(0"&,&``0(0"&\mbox{and}&``1)1"&,&\mbox{ respectively, thus yielding a string }f''.
 \end{array}$$
Add a terminal parenthesis to $f''$, so that the last ``$1$" in $f''$ is transformed into ``$1)$". Denote by $g$ the string resulting from such addition of a closing parenthesis to $f''$.
\item
By proceeding from left to right, replace the bits of $g$ by the successive integers from 1 to $|g|$, keeping all pre-inserted parentheses and commas in $g$ unchanged in  position. This yields a version $h_0$ of $g$ from which removal of parenthesis and commas yields $F'$.
 \item
 Present $h_0$ as a concatenation $(w_1)|(w_2)|\cdots|(w_t)$ of expressions $(w_i)$, ($i=1,\ldots,t$), for adequate $t\ge 1$, each $(w_i)$ with terminal ``)" being the closing ``)" nearest to its opening  ``(".  Let $w'_i$ be the number string obtained from $w_i$ by the removal of all its parentheses and commas. For $i=1,\ldots,t$,
perform a recursive step ${\mathbf{•}cal R}$ consisting first in transforming $w'_i$ into its reverse substring $w''_i$ and then resetting $w''_i$ instead of $w'_i$ in $(w_i)$, with the parentheses and commas of $(w_i)$ kept unchanged. Denote the resulting expression by ${\mathcal R}(w_i)$.
Ortherwise, let ${\mathcal R}^2=$. This yields a string $h_1={\mathcal R}(w_1)|{\mathcal R}(w_2)|\cdots|{\mathcal R}(w_t).$

\item For $i\in[1,t]$, let ${\mathcal R}(w_i)=(a_{i,1}^1\eta_{i,1}^1b_{i,1}^1)|(a_{i,2}^1\eta_{i,2}^1b_{i,2}^1)|(a_{i,3}^1\cdots)|(\cdots b_{i,t_i-1}^1)|(a_{i,t_i}^1\eta_{i,t_i}^1b_{i,t_i}^1)$, each $\eta_{i,j}^1\ne\epsilon$ ($\epsilon$ representing a comma ``," in $h_1$, so $b_{i,j}^1=a_{i,j}^1\pm 1$, if $\eta_{i,j}^1=\epsilon$) of the form $\eta_{i,j}^1=(w_{i,j})$,
with terminal ``)" being the closing ``)" nearest to its opening  ``(", for $j\in[1,t_i]$. Apply item 3 to each $(w_{i,j}^1)$ in place of $h_0$, for $j\in[1,t_i]$. Replace the resulting strings ${\mathcal R}(w_{i,j})$ in the positions of the corresponding $(w_{i,j})$ in ${\mathcal R}(w_i)$, yielding a modified version ${\mathcal R}^2(w_i)$ of ${\mathcal R}(w_i)$. Let $h_2={\mathcal R}^2(w_1)|{\mathcal R}^2(w_2)|\cdots|{\mathcal R}^2(w_t)$.
\item
Each ${\mathcal R}(w_{i,j})$ is a concatenation of terms of the form $a^2_I\eta^2_Ib^2_I$, with $I=\{i,j_1,j_2\}$ by letting $j_1=j$. In each such concatenation, the strings $\eta^2_I\ne\epsilon$ are of the form $(w_I)$ and must be treated like $(w_{i,j})$ is in item 4 (or $(w_i)$ in item 3), producing a modified string ${\mathcal R}(w_I)$.
Proceeding this way,
a sequence of strings $(h_0,h_1,h_2,\ldots,h_{s+1})$ is eventually obtained for some $s\ge 0$ when all  innermost expressions of the form $(w_I)=(a,a\pm 1)$ with $a,a\pm 1\in[1,2k]$) are finally processed, where $I={i,j_1,\ldots,j_s}$.
\item
 Disregarding the parentheses and commas in $h_{s+1}$ yields a $2k$-string $g'$ that insures an assignment $i\rightarrow p(i)$, for $i=1,\ldots,2k$, by making correspond the positions $i=1,\ldots,2k$ of $g'$ to the actual entry values in those positions of $g'$.
 \item
Define $\pi=p^{-1}$, the inverse $2k$-permutation of $p=(p(1)p(2)\cdots p(2k))$ \cite{Hcs}.
\end{enumerate}

\begin{example}\label{Inc} {\bf(Continuation of Example~\ref{Toc})} The middle right of Fig.~\ref{fig4} (just under the upper-right representation of the curve $PLC(\alpha_1)$) contains a list, call it $\ell$, whose first line represents $g(\alpha_1)$ (Subsection~\ref{rev}, item 1), with $\alpha_1$ as in Example~\ref{Toc}, and whose second line represents $h_0(\alpha_1)$ (Subsection~\ref{rev}, item 2), in an hexadecimal-notation continuation. In this representation of $h_0(\alpha_1)$, the red substrings $w_1=``1(2,3)4''$, $w_2=``5,6''$ and $w_3=``7(8(9(a,b)c)d)(e,f)(g,h,i)j)k''$ are to be reversed according to the first instances of ${\mathcal R}(w)$ in Subsection~\ref{rev}. This yields the third line, representing $h_1(\alpha_1)$ in the list $\ell$. In $h_1(\alpha_1)$, the red substrings are to be reversed according to the next instances of ${\mathcal R}(w)$, and so on. In the end, the sixth line of $\ell$, represents $h_4(\alpha_1)=$ $$p(\alpha_1)=(4,3,2,1,6,5,20,14,18,16,17,15,19,12,13,8,10,9,11,7).$$
The inverse of this is $\pi(\alpha_1)=$ $$p^{-1}(\alpha_1)=(4,3,2,1,6,5,20,16,18,17,19,14,15,8,12,10,11,9,13,7),$$ represented as a blue string under the mentioned sixth line $h_4(\alpha_1)=p(\alpha_1)$ and in a similar format with inserted parentheses and commas.
\end{example}

\begin{example}\label{rrr2} {\bf(Continuation of Example~\ref{rr2})} Fig.~\ref{fig3} contains one oriented 3-cycle for $O_1$ and two oriented 5-cycles for $O_2$. Their lists $L(\alpha)$ are headed by two lines: a first line reading ``$ord(\alpha)$:$F(\alpha);\pi(\alpha)$", with ``$F(\alpha)$'' as the first line of the cycle and ``$\pi(\alpha)$" (as in Subsection~\ref{rev}), formed by the different entries at which a blue-to-red $k$-supplementation takes place in the cycle; the second line contains the (underlined) positions 0 to $2k$ of the vertices (as $n$-tuples) in the cycle, followed by ``$O_k$". The arcs of $O_k$ receive colors in the set $[0,k]$ so that the edge between each two adjacent vertices in those cycles has its two composing arcs bearing $k$-{\it supplementary} colors $b$ (for blue) and $r$ (for red), meaning that $b,r\in[0,k]$ are such that $b+r=k$.
To the immediate right of each of these three cycles, for lists $L(\epsilon),L(0),L(1)$ of respective lengths 3, 5, 5, are also represented vertical lists $L^M(\epsilon), L^M(0), L^M(1)$, (occupying two contiguous columns each) closing into corresponding cycles $C^M(\epsilon), C^M(0), C^M(1)$ of respective double lengths 6, 10, 10, obtained by replacing the ``$=$" signs by the ``$>$" signs  and ``$<$" signs  uniformly on alternate lines.
These cycles can be interpreted as cycles in the middle levels graphs $M_1,M_2$, obtained by reading the subsequent lines in the concatenation of two subsequent columns as follows: from left to right if they bear ``$>$'' signs, and from right to left if they bear ``$<$'' signs. In addition, the graphs $O_1$, $O_2$, $M_1$, $M_2$ are represented in Fig.~\ref{fig3} in thick trace for the edges containing the arcs of the oriented cycles $C(\alpha)$; recalling $\mathbb{A}_k$ from Section~\ref{s1}, each
vertex (resp., edge) of $O_1$, $O_2$ is denoted by the support of its corresponding bitstring $f(\alpha)$
(resp., denoted centrally by its underlined color in $\mathcal{E}_k$ and marginally by its blue-red arc-color pair in $\mathbb{A}_k$). In $M_1$, $M_2$, a plus or minus sign precedes each such support indicating respectively a vertex in level $L_k$ or in level $L_{k+1}$ of $B_n$; if in $L_{k+1}$, as the complement $\overline{f(\alpha,<)}$ of the right-to-left reading $f(\alpha,<)$ of the bitstring $f(\alpha)=f(\alpha,>)$; if in $L_k$, as $f(\alpha)=f(\alpha,>)$ itself. The resulting readings of $n$-tuples of $M_1$, $M_2$ inherit the mentioned arc colors for $O_1$, $O_2$, corresponding to the {\it modular matchings} of \cite{DKS}, only that the colors in \cite{DKS} are in $[1,k+1]$ with supplementary sum $k+1$ while our colors are in $[0,k]$ with supplementary sum $k$.\end{example}

\section{Uniform 2-factors and Hamilton cycles}\label{verL}

Let $k>1$. A vertical list $L(\alpha)$ as in Examples~\ref{rr2} and \ref{rrr2}, illustrated in Fig~\ref{fig3}, can be formed for each $k$-germ $\alpha$.
In fact, there are $C_k$ such lists $L(\alpha)=(L_0(\alpha),L_1(\alpha),\ldots,L_{2k}(\alpha))^t$, where $t$ stands for transpose, each $L(\alpha)$ representing in $O_k$ an oriented $n$-path $P(\alpha)$ whose end-vertices $L_0(\alpha)=F(\alpha)$ and $L_{2k}(\alpha)$, are adjacent in $O_k$, thus completing an oriented cycle $C(\alpha)$ in $O_k$ by the addition of the arc $(L_{2k}(\alpha),L_0(\alpha))$.

\begin{figure}[htp]
\includegraphics[scale=0.75]{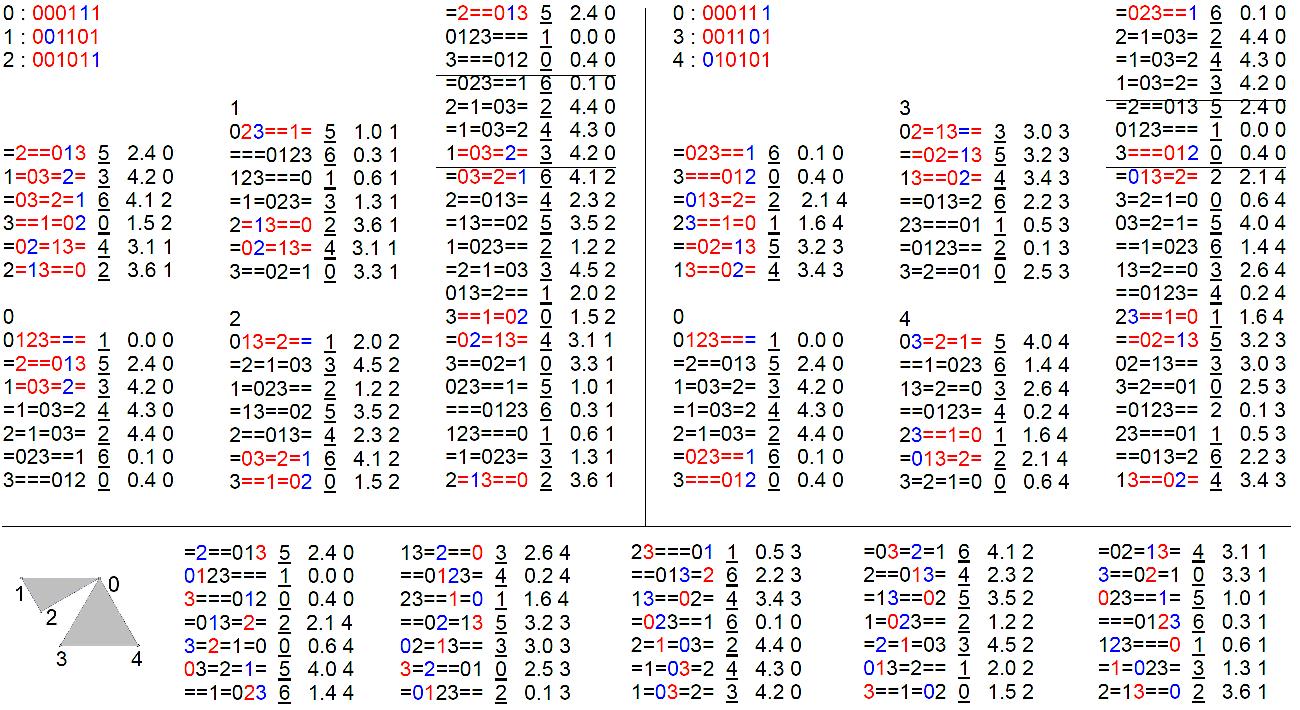}
\caption{Illustration for Section~\ref{verL} and Example~\ref{ex6}}
\label{fig6}
\end{figure}

Those paths $P(\alpha)$ arose in \cite[Theorem 4]{u2f} and \cite[Lemma 4]{Hcs}, in the latter case leading to Hamilton cycles in $O_k$ and $M_k$. Construction of these $L(\alpha)$ is controlled by the $2k$-permutation $\pi(\alpha)$ assigned to $\alpha$ via the procedure contained in items 1--7 of Subsection~\ref{rev}, as will be established in Theorem~\ref{L5}.

\subsection{Flippable tuples and flipping cycles}\label{joder}

Fig.~\ref{fig6} for $k=3$ and Fig.~\ref{fig7} for $k=4$ (Example~\ref{ex7}) contain the lists $L(\alpha)$ assembled in triples and/or quadruples.
For $k=3$, Fig.~\ref{fig6} shows two such triples, that we call $\tau_0=(L(\alpha_0),L(\alpha_1),L(\alpha_2))$  on the upper-left of the figure and $\tau_1=(L(\alpha_0),L(\alpha_3),L(\alpha_4))$ on the upper-right, where each list $L(\alpha_i)$ is distinguished on its upper-left corner with the subindex $i$ of its $k$-germ $\alpha_i$. Two concepts from \cite{Hcs} are used here:

\begin{figure}[htp]
\hspace*{1.1cm}
\includegraphics[scale=1.05]{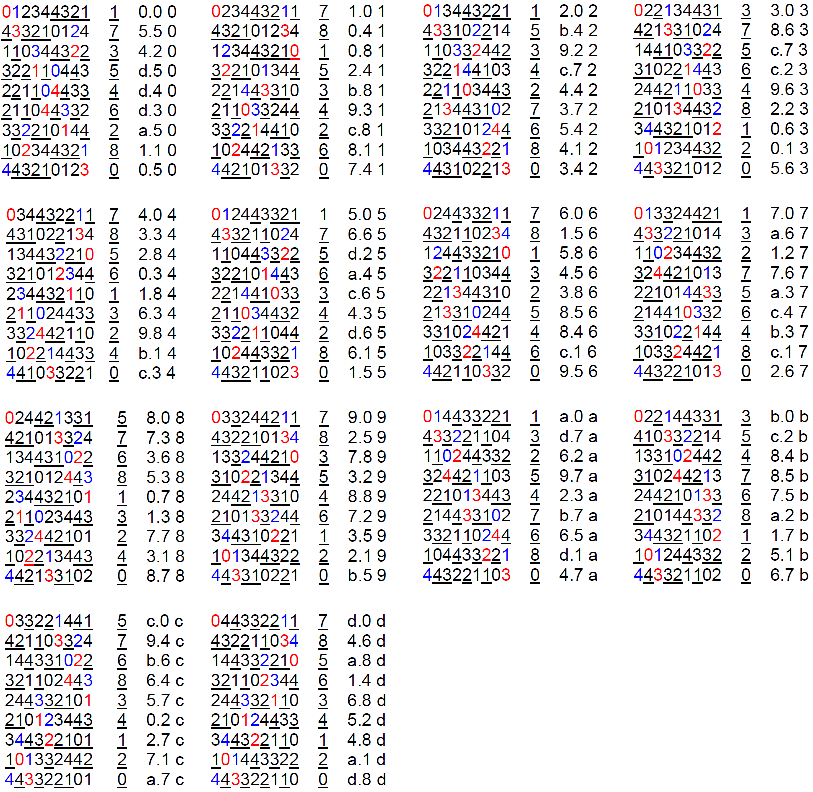}
\caption{Illustration for Subsection~\ref{uu} and Example~\ref{ex7}}
\label{fig5}
\end{figure}

\begin{enumerate}\item
{\bf Flippable tuples:} In each such $L(\alpha_i)$, there is at least one pair of contiguous red lines, apart from, or including, its first red line, $F(\alpha_i)$, except for their initial black entries and the unique vertical pair of number  $k$-supplementary blue entries (Section~\ref{modular}).
 These unique colored-line pairs $FT(\alpha_i,j)$, where $j\in[0,2k]$ are the respective positions counted from the right at which the $k$-supplementary pairs determining adjacency in $O_k$ occurs,  are said to be {\it flippable tuples} \cite{Hcs}.
 \item
{\bf Flipping $\kappa$-cycles:} For $j=0,1$, the three pairs $FT(\alpha_i,j)$ with $L(\alpha_i)\in\tau_j$ are combined into a 6-cycle $F_6C(\tau_j)$ in $O_3$, which for $\kappa=2k=6$, is an example of a {\it flipping $\kappa$-cycle} \cite{Hcs}, that we denote $F_\kappa C(\tau_j)$. Such flipping 6-cycle $F_6C(\tau_j)$ is shown in the middle left of the respective upper-left part and upper-right part, respectively, in
 Fig.~\ref{fig6}, sided each on its right and below by its three participating lists $L(\alpha_i)$.
 Above such flipping 6-cycle $F_6C(\tau_j)$ ($j=0,1)$, a triple of Dyck words of length 6 headed each by the subindex $i\in\{0,1,2\}$ or $i\in\{0,3,4\}$ of the corresponding $L(\alpha_i)$ is shown in red except for one blue entry at the position of the blue $k$-supplementary number entries in the two contiguous red-blue substrings of the corresponding flippable tuples
$FT(\alpha_i,j)$.
\end{enumerate}
For any $k>1$, red-blue flippable tuples as those six in this subsection (see Fig~\ref{fig6}) were shown to exist in \cite[pp.~1261--1265]{Hcs}. These six flippable tuples were shown to form part of a bitstring family \cite[display~(3.3)]{Hcs} (see displays~(\ref{!}) and~(\ref{!!}) in Subsection~\ref{r4} below). They were used in the construction of Hamilton cycles in \cite{Hcs}, reconsidered below.

\begin{example}
For $k=3$, Fig.~\ref{fig6} contains, on the right of each of the two cases of $\tau_j$ in Subsection~\ref{joder}, the symmetric difference of the corresponding flipping 6-cycle $F_6C(\tau_j)$ and the union of the three 7-cycles $C(\alpha_i)$, for each $i\in(0,1,2)$ or $i\in(0,3,4)$, yielding a 21-cycle in each case.
The two 21-cycles are then recombined into a Hamilton cycle of $O_3$, shown on the lower part of Fig.~\ref{fig6} as a list sectioned from left to right into five sublists. To the left of these five sublists, there is a drawing of an hypergraph as defined in Subsection~\ref{r4} below. \end{example}

\subsection{Modified n-tuples}\label{uu}

Each $n$-tuple $F(\alpha)$ gives place to a modified $n$-tuple $\underline{F}(\alpha)$ formed by the number entries $j\in[0,k]$ of $F(\alpha)$ set in the same positions they have in $F(\alpha)$ together with $k$ underlined number entries $\underline{j}$ in place of the ``$=$" signs, where $j\in[1,k]$ (or $\underline{j}\in\{\underline{1},\ldots,\underline{k}\}$), in a fashion determined by the fact that a nonempty Dyck word is expressible uniquely as a string $0u1v=0_u^vu1_u^vv$ (modified from $1u0v$ \cite[p.~1260]{Hcs}), where $u$ and $v$ are (possibly empty) Dyck words.
Each number entry $j\in[0,k]$ in $F(\alpha)$ corresponds to the starting entry $0_u^v$ of a Dyck word $0_u^vu1_u^vv$ in $f(\alpha)$, with its $1_u^v$ represented in $F(\alpha)$ by an ``$=$" sign. Its $\underline{F}(\alpha)$ has each number entry $j$ $(\ne\underline{j})$ in its same position as in $F(\alpha)$, with a corresponding entry $0_u^v$ of a Dyck word $0_u^vu1_u^vv$ in $f(\alpha)$.

Moreover, $\underline{F}(\alpha)$ has each ``$=$" sign of $F(\alpha)$ replaced by a corresponding underlined integer $\underline{j}$ in the position of an accompanying $1_u^v$. As an example, the right side of Fig.~\ref{fig4} contains, under the list $\ell$ of Example~\ref{Inc} and the blue string containing $\pi(\alpha_1)=p^{-1}(\alpha_1)$, a red line repeating the first line $g(\alpha_1)$ of $\ell$, and a subsequent red line with the 0-bits and 1-bits of $g(\alpha_1)$ replaced by the respective number entries $j$ and $\underline{j}$ of $\underline{F}(\alpha_1)$.

For $k=4$, Fig.~\ref{fig5} contains vertical lists $\underline{L}(\alpha)=(\underline{L}_0(\alpha),\underline{L}_1(\alpha),\ldots,\underline{L}_{2k}(\alpha))^T$ similar to the lists $L(\alpha)$ but corresponding instead to the $n$-strings $\underline{F}(\alpha)=\underline{L}_0(\alpha)$, $\underline{L}_1(\alpha)$, $\ldots$, $\underline{L}_{2k}(\alpha)$, where $\alpha$ runs over the total of fourteen $4$-germs and the only non-black entries are those corresponding to the 4-supplementary vertical blue-red pairs realizing the adjacency of each pair of contiguous lines, including the pair formed by the initial blue ``4" in the last line $\underline{L}_8(\alpha)$ and the initial red 0 in the first line $\underline{L}_0(\alpha)$ in each list. All the first columns of the fourteen lists form the same column vector, with transpose row vector $(0,\underline{4},1,\underline{3},2,\underline{2},3,\underline{1},4)$.

The sole representative $f(\alpha)$ of a $\mathbb{Z}_n$-class of $V(O_k)$, as in Theorem~\ref{bij}, may not only be interpreted as the $n$-tuple $F(\alpha)$ but also as the corresponding $\underline{F}(\alpha)$, so the other $n$-tuples of that class may be interpreted as its translations mod $n$. The lines of each $L(\alpha)$ and the lines of its associated $\underline{L}(\alpha)$ are translations mod $n$ of respective $n$-tuples $F(\alpha_\iota)$ and $\underline{F}(\alpha_\iota)$ that depend on the orders $\iota\in[0,2k]$ of such lines. These facts are used in the statement of Theorem~\ref{L5}, where the subindex $j$ is $j=2k-\iota$ in relation to the subindex $\iota$.

\section{Iterative generation of modified n-tuples}

\begin{theorem}\label{L5} For each $k$-germ $\alpha$:
\begin{enumerate}
\item[\bf(i)] $L(\alpha)$ is generated by transforming iteratively
for $j=2k,2k-1,\ldots,2,1$ and with initial $n$-tuple $L_0(\alpha)=F(\alpha)$ the $n$-tuple $L_{2k-j}(\alpha)$ into the uniquely feasible next $n$-tuple, $L_{2k-j+1}(\alpha)$, via $k$-supplementation of its $\pi(j)$-th entry and exchange of its $k$ remaining number entries by its $k$ ``$=$" sign entries;
\item[\bf(ii)] the first column of $\underline{L}(\alpha)$ has transpose row vector  $$(0,\underline{k},1,\underline{k-1},2,k-2,\cdots,\underline{3},k-2,\underline{2},k-1,\underline{1},k),$$ obtained by alternating the entries of the vectors $$(0,1,2,\ldots,k-1,k)\mbox{   and   }(\underline{k},\underline{k-1},\ldots,\underline{2},\underline{1});$$ moreover, $k\underline{k}$ and $\underline{1}0$ are substrings$\mod{n}$ of each $\underline{L}_j(\alpha)$.
\end{enumerate} The resulting lists $L(\alpha)$ and $\underline{L}(\alpha)$, yield a uniform 2-factor of $O_k$ formed by $C_k$ $n$-cycles.
\end{theorem}

\begin{proof} Item (i) is an adaptation of \cite[Lemma 5]{Hcs} to the $k$-germ setting of Subsections~\ref{alfalfa}-\ref{Dyck} and~\ref{rev}-\ref{uu} as well as the following argument.

The Dyck path of length $2k$ defined in Subsection~\ref{Dyck} corresponds to the Dyck paths with $2k$ steps and 0 flaws of \cite{u2f}, presented in each list $L(\alpha)$ as $L_0(\alpha)$.

In the same way, $L_2(\alpha),L_4(\alpha),\ldots,L_{2k}(\alpha)$ correspond respectively to the Dyck paths with $2k$ steps and $1,2,\ldots,k$ flaws of \cite{u2f}, obtained in our cases again as in Subsection~\ref{rev} by the removal of its first up-step and change of coordinates from $(1,1)$ to $(0,0)$.

In fact, passing from each $L_{2i}(\alpha)$ to $L_{2i+1}(\alpha)$ corresponds to applying the function $g$ defined in the second paragraph of \cite[Subsection 1.1]{u2f}.

Passing from $L_{2i+1}(\alpha)$ to $L_{2k+2}(\alpha)$ corresponds to applying the function\ $h$ composing the mapping $f=h\circ g$ of Theorem 2 \cite{u2f}.

For item (ii), note that the $n$-tuples $\underline{L}_j(\alpha)$ having a common initial entry in $[0,k]\cup[\underline{1},\underline{k}]$ are at the same height $j$ in all vertical lists $\underline{L}(\alpha)$
so that the entries of the first column $(a_0,\underline{b}_0,a_1,\underline{b}_1,\ldots,a_k,\underline{b}_k,a_{k+1})^T$ of each such $\underline{L}(\alpha)$ satisfy both $a_i+b_i=k$ and $b_i+a_{i+1}=k+1$, for $i\in[0,k]$.

 Thus, the alternating first-entry column in each vertical list characterizes and controls the formation of the claimed uniform 2-factor.
\end{proof}

\subsection{Dyck-word triples and quadruples}\label{r4}

Consider the following Dyck-word collections (triples, quadruple, etc.):
\begin{eqnarray}\label{!}\begin{array}{lllllr}
S_1(w)&=&\{\xi_{1(w)}^1=0w001\underline{1}1,&\xi_{1(w)}^2=0w\underline{0}1101,&\xi_{1(w)}^3=0w0101\underline{1}&\},\\
S_2&=& \{\xi_2^1\hspace{4.4mm}=0\underline{0}110011,&\xi_2^2\hspace{4.4mm}=0010011\underline{1},&\xi_2^3\hspace{4.4mm}=00010\underline{1}11&\},\\
S_3&=& \{\xi_3^1\hspace{4.4mm}=00011\underline{1},&\xi_3^2\hspace{4.4mm}=0100\underline{1}1,&\xi_3^3\hspace{4.4mm}=\underline{0}10101&\},\\
S_4&=&\{\xi_4^1\hspace{4.4mm}=00011\underline{1},&\xi_4^2\hspace{4.4mm}=0010\underline{1}1,&\xi_4^3\hspace{4.4mm}=01\underline{0}011,&\xi_4^4=\underline{0}10101\},
\end{array}\end{eqnarray}

\noindent (based on \cite[display~(4.2)]{Hcs}) where $w$ is any (possibly empty) Dyck word. Consider also the sets $\underline{S}_1(w)$, $\underline{S}_2$, $\underline{S}_3$, $\underline{S}_4$ obtained respectively from $S_1(w)$, $S_2$, $S_3$, $S_4$ by having their component Dyck paths
$\underline{\xi}_{1(w)}^j$, $\underline{\xi}_2^j$, $\underline{\xi}_3^j$, $\underline{\xi}_4^j$
defined as the complements of the reversed strings of the corresponding Dyck paths
$\xi_{1(w)}^j$, $\xi_2^j$, $\xi_3^j$, $\xi_4^j$.
Note that each Dyck word in the subsets of display~(\ref{!}) has an underlined entry. By denoting
\begin{eqnarray}\label{denot}\xi_{i(w)}^j=x_sx_{s-1}\cdots x_2x_1x_0\mbox{  and  }\xi_i^j=x_sx_{s-1}\cdots x_2x_1x_0,\mbox{ for }i=2,3,4,\end{eqnarray} where $j=1,2,3$ for $i=1,2,3$ and $j=1,2,3,4$ for $j=4$ and adequate $s$ in each case, the underlined positions in~(\ref{!}) are the targets of the following correspondence $\Phi$:
\begin{eqnarray}\label{!!}\begin{array}{llll}
\Phi(\xi_{1(w)}^1)=1,&\Phi(\xi_{1(w)}^2)=4,&\Phi(\xi_{1(w)}^3)=0,&\\
\Phi(\xi_2^1)\hspace{4.4mm}=6,&\Phi(\xi_2^2)\hspace{4.4mm}=0,&\Phi(\xi_2^3)\hspace{4.4mm}=2,&\\
\Phi(\xi_3^1)\hspace{4.4mm}=0,&\Phi(\xi_3^1)\hspace{4.4mm}=1,&\Phi(\xi_3^3)\hspace{4.4mm}=5,&\\
\Phi(\xi_4^1)\hspace{4.4mm}=0,&\Phi(\xi_4^2)\hspace{4.4mm}=1,&\Phi(\xi_4^3)\hspace{4.4mm}=3,&\Phi(\xi_4^4)=5.\\
\end{array}\end{eqnarray}
The correspondence $\Phi$ is extended over the Dyck words $\underline{\xi}_{1(w)}^j$, $\underline{\xi}_2^j$, $\underline{\xi}_3^j$, $\underline{\xi}_4^j$ with their barred positions taken reversed with respect to the corresponding barred positions in $\xi_{1(w)}^j$, $\xi_2^j$, $\xi_3^j$, $\xi_4^j$, respectively.

Recall the ordered tree ${\mathcal T}_k$ from Theorem~\ref{thm1}. Adapting from \cite{Hcs}, we define an hypergraph $H_k$ with $V(H_k)=V({\mathcal T}_k)$ and having as hyperedges the subsets $\{\alpha^j;j\in\{1,2,3\}\}\subset V(H_k)$ and $\{\alpha^j;j\in\{1,2,3,4\}\}\subset V(H_k)$ whose member $k$-germs $\alpha^j$ have associated bitstrings $f(\alpha^j)$, for $j=1,2,3\mbox{ or }j=1,2,3,4$, containing respective Dyck words in $\{\xi_{1(w)}^j$, $\xi_2^j$, $\xi_3^j$, $\xi_4^j$, $\underline{\xi}_{1(w)}^j$, $\underline{\xi}_2^j$, $\underline{\xi}_3^j$, $\underline{\xi}_4^j\}$ in the same 6 or 8 fixed positions $x_i$ (for specific indices $i\in\{0,1,\ldots,s\}$ in~(\ref{denot})) and forming respective subsets
$\{\xi_{1(w)}^j(w);j=1,2,3\}$,
$\{\underline{\xi}_{1(w)}^j;j=1,2,3\}$,
$\{\xi_4^j;j=1,2,3,4\}$,
$\{\underline{\xi}_4^j;j=1,2,3,4\}$,
$\{\xi_i^j;j=1,2,3\}$ and
$\{\underline{\xi}_i^j;j=1,2,3\}$, for both $i=2$ and $3$.

\begin{example}\label{ex6} Two hyperedges $h_0,h_1$ of $H_3$ are shown heading on the upper-left and upper-right sides of Fig.~\ref{fig6}, respectively, with strings $\xi_{1(\epsilon)}^j$ or $\xi_i^j$ ($j=3,4$) having their constituent entries in red except for one barred entry, in blue. For $h_0$ (resp., $h_1$), $f(\alpha_0)$,
$f(\alpha_1)$, $f(\alpha_2)$, (resp., $f(\alpha_0)$, $f(\alpha_3)$, $f(\alpha_4)$), represented by the respective subindices 0, 1, 2 (resp., 0, 3, 4), are shown stacked in the upper-left (-right) of the figure, those subindices indicating (each via a colon) respectively the Dyck words $\xi_{1(\epsilon)}^1$, ${\underline{\xi}}_3^2$, $\xi_{1(\epsilon)}^3$ (resp., $\xi_{1(\epsilon)}^3$, $\xi_3^2$, $\xi_4^3$), with their entries in red except for the entries in positions $\Phi(\xi_{1(\epsilon)}^1)=1$, $\Phi({\underline{\xi}}_3^2)=4$, $\Phi(\xi_{1(\epsilon)}^3)=0$ (resp., $\Phi(\xi_{1(\epsilon)}^3)=0$, $\Phi(\xi_3^2)=1$, $\Phi(\xi_4^3)=5$), which are blue. Then $H_3$ contains the connected subhypergraph $H_3'$ depicted
on the lower-left of the figure. This is used to construct the shown Hamilton cycle. The hyperedges of $H'_3$ are denoted by the triples of subindices $i$ of their composing 4-germs $\alpha_i$. So, the hyperedges of $H'_3$ are taken to be $h_0=(0,1,2)$ and $h_1=(0,3,4)$. This type of notation is used in Example~\ref{ex7}, as well.
\end{example}

\begin{figure}[htp]
\hspace*{1.8cm}
\includegraphics[scale=1.28]{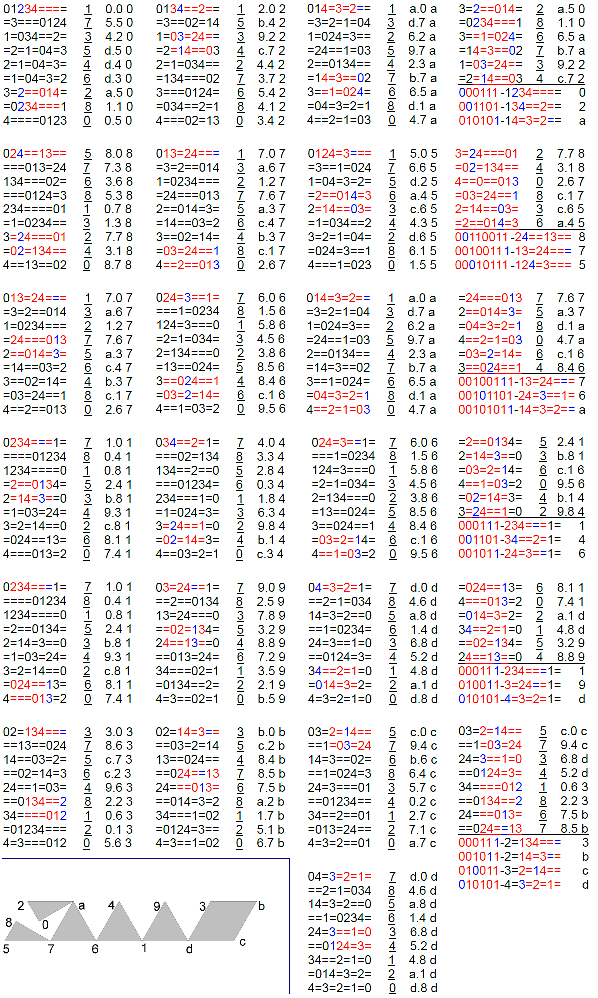}
\caption{Illustration for Section~\ref{verL} and Example~\ref{ex7}}
\label{fig7}
\end{figure}

\begin{example}\label{ex7} For $k=4$, let us represent each $k$-germ $\alpha_i$ by its respective order $i=ord(\alpha_i)$.
In a likewise manner to that of Example~\ref{ex6}. Fig.~\ref{fig7} shows on its lower-left corner a depiction of a subhypergraph $H'_4$ of $H_4$ with the hyperedges
$$h_0=(0,2,a),\; h_1=(8,7,5),\; h_2=(7,6,a),\; h_3=(1,4,6),\; h_4=(1,9,d),\; h_5=(3,b,c,d).$$
The respective triples of Dyck words $\xi_{1(w)}^j$ or $\underline{\xi}_{1(w)}^j$ or $\xi_i^j$ or $\underline{\xi}_i^j$ ($j=2,3,4$) may be expressed as follows by replacing the Greek letters $\xi$ by the values of the correspondence $\Phi$:
$$
(\underline{5}_3^1,\underline{4}_3^2,\underline{0}_3^3),\;
(6_2^1,0_2^2,2_2^3), \;
(1_{1(01)}^1,4_{1(01)}^2,0_{1(01)}^3),\;
(1_{1(\epsilon)}^1,4_{1(\epsilon)}^2,0_{1(\epsilon)}^3),\;
(0_3^1,1_3^2,5_3^3),\;
(0_4^1,1_4^2,3_4^3,5_4^4),$$
where we can also write $(\underline{5}_3^1,\underline{4}_3^2,\underline{0}_3^3)=\underline{(0_3^1,1_3^2,5_3^3)}$.
From top to bottom in Fig.~\ref{fig7}, excluding the said depiction of $H'_4$, the vertical lists corresponding to the composing $4$-germs of those six hyperedges are presented side by side, in a fashion similar to that of Fig.~\ref{fig6}, except that the first line in each such vertical list has its
corresponding substring $\xi$ (a member of one of the sets presented in display~(\ref{!!})) in red but for its blue entry
$\Phi(\xi)$ to stress their roles in the respective $L(\alpha)$ and $FT(\alpha,j)$. The flippable tuples $FT(\alpha,j)$  allow to compose five flipping $6$-cycles and one flipping $8$-cycle, presented to the right of each triple or quadruple of vertical lists, allowing to integrate, by symmetric differences, a Hamilton cycle comprising all the vertices in the cycles provided by the vertical lists. Below those $6$- or $8$-cycles, the corresponding red-blue substrings $\xi_i^j$ appear separated by a hyphen in each case from the associated (multicolored) first lines.

We represent $H_k$ as a simple graph $\psi(H_k)$ with $V(\psi(H_k))=V(H_k)$ by replacing each hyperedge $e$ of $H_k$ by the clique $K(e)=K(V(e))$ so that $\psi(H_k[e])=K(e)$, being such replacements the only source of cliques of $\psi(H_k)$. A {\it tree} $T$ of $H_k$ is a subhypergraph of $H_k$ such that: {\bf(a)} $\psi(T)$ is a connected union of cliques $K(V(e))$; {\bf(b)} for each cycle $C$ of $\psi(H_k)$, there exist a unique clique $K(V(e))$ such that $C$ is a subgraph of $K(e)$. A {\it spanning tree} $T$ of $H_K$ is a tree of $H_k$ with $V(T)=V(H_k)$. Clearly, the subhypergraphs $H'_k$ of $H_k$ depicted in Fig.~\ref{fig6} and~\ref{fig7} for $k=3$ and 4 are corresponding spanning trees.

A subset $G$ of hyperedges of $H_k$ is said to be {\it conflict-free} \cite{Hcs} if: {\bf(a)} any two hyperedges of $G$ have at most one vertex in common; {\bf(b)} for any two hyperedges $g,g'$ of $G$ with a vertex in common, the corresponding images by $\Phi$ (as in display~(\ref{!!})) in $g$ and $g'$ are distinct. A proof of the following final result is included, as our viewpoint and notation differs from that of \cite{Hcs}.
\end{example}

\begin{theorem}\label{L6} (\cite{Hcs}) A conflict-free spanning tree of $H_k$ yields a Hamilton cycle of $O_k$, for every $k\ge 3$. Moreover, distinct conflict-free spanning trees of $H_k$ yield distinct Hamilton cycles of $H_k$, for every $k\ge 6$.
 \end{theorem}

\begin{proof} Let $D_k$ be the set of all Dyck words of length $2k$ and, recalling display (\ref{!}), let
\begin{eqnarray}\label{rec1}\begin{array}{|l|l|l|}\hline
E_2=\{0101\} \!&\! E_3=S_4 \!&\! E_k=01D_{k-1}, \forall k>3\\
F_2=\{0011\} \!&\! F_3=D_3\setminus E_3=\{001101\}					     \!&\!F_k=D_k\setminus 01D_{k-1}, \forall k>3\\\hline
\end{array}\end{eqnarray}
In particular, $0101(01)^{k-2}\in E_k$ and $0011(01)^{k-2}\in F_k$. Now, let
\begin{eqnarray}\label{rec2}\begin{array}{|l|l|l|l|}\hline
{\mathcal E}_2=\emptyset & \mathcal{E}_3=\{S_4\}       & \mathcal{T}_3=\{S_1(\epsilon),S_3\}&\mathcal{E}_k=01\mathcal{T}_{k-1},  \forall k>3\\
{\mathcal F}_2=\emptyset & \mathcal{F}_3=\emptyset & \mathcal{F}_4=\{S_1(01),S_2,0\underline{S}_31,S_1(\epsilon)01\}&\\\hline
\end{array}\end{eqnarray}
Let us set $\mathcal{F}_k$ as a function of $\mathcal{E}_2,\ldots,\mathcal{E}_{k-1},\mathcal{F}_2,\ldots,\mathcal{F}_{k-1},\mathcal{T}_{k-2}$, as follows:
For $1<j\le k$, let $F_k^j=\cup_{i=2}^j\{0\underline{u}1v;u\in D_{i-1},v\in D_{k-1}\}$. Since $F_k=F_k^k$, then the following implies the existence of a spanning tree of $H_k[F_k]$.

\begin{lemma} For every $1<j\le k$, there exists a spanning tree $\mathcal{F}_k^j$ of $H_k[F_k^j]$.
\end{lemma}

\begin{proof}
Lemma 7 \cite{Hcs} asserts that if $\tau$ is a flippable tuple and $u,v$ are Dyck words, then: {\bf(i)} $u\tau v$ is a flippable tuple if $|u|$ is even; {\bf(ii)} $u\underline{\tau}v$ is a flippable tuple if $|u|$ is odd. Lemma 8 \cite{Hcs} insures that the triples and quadruples in (\ref{!}) are flippable tuples. Using those two lemmas of \cite{Hcs}, we define
$\Psi$ as the set of all such flippable tuples $u\tau v$ and $u\underline{\tau}v$.
Moreover, we define $\Psi_2=\emptyset$ and $\Psi_k=\Psi\cap D_k$, for $k>2$.

Since $F_k^2=0011D_{k-2}$, we let $\mathcal{F}_k^2=0011\mathcal{T}_{k-2}$. Assuming $2<j\le k$, since $D_{j-2}=E_{j-1}\cup F_{j-1}$ is a disjoint union, then we have the following partition:
\begin{eqnarray}\label{quelio}F_k^j=F_k^{j-1}\cup_{v\in D_{k-j} }(0\underline{D}_{j-1}1v)=F_k^{j-1}\cup_{v\in D_{k-j} }((0\underline{E}_{j-1}1v)\cup(0\underline{F}_{j-1}1v)).\end{eqnarray}
For every $v\in D_{k-j}$, the elements of $\tau(v)=S_1((01)^{j-3})v\in\Psi_k$ are:
\begin{eqnarray}\label{rec3}\begin{array}{|c|c|c|}\hline
0(01)^{j-3}001\underline{1}1v\in 0\underline{F}_{j-1}1v & 0(01)^{j-3}0101\underline{1}v\in 0\underline{E}_{j-1}1v & 0(01)^{j-3}\underline{0}1101v\in F_k^{j-1}\\\hline\end{array}\end{eqnarray}
Now, we let
\begin{eqnarray}\label{formula}\mathcal{F}_k^j=\mathcal{F}_k^{j-1}\cup(\cup_{v\in D_{k-j}}(\{\tau(v)\}\cup (0\underline{\mathcal{E}}_{j-1}1v)\cup (0\underline{\mathcal{F}}_{j-1}1v))),\end{eqnarray} which defines a spanning tree of $H_k[F_k^j]$. \end{proof}

Now, the elements of $\tau=S_3(01)^{k-3}\in\Psi_k$ are:
\begin{eqnarray}\label{rec4}\begin{array}{|c|c|c|}\hline
00011\underline{1}(01)^{k-3}\in F_k$, ($k>3) & 0100\underline{1}1(01)^{k-3}\in 01E_{k-1} & \underline{0}10101(01)^{k-3}\in 01F_{k-1}\\\hline\end{array}\end{eqnarray}

The sets $F_k$, $01E_{k-1}$ and $01F_{k-1}$ form a partition of $D_k$. We take the spanning trees of the subhypergraphs induced by these three sets and connect them into a single spanning tree of $H_k$ by means of the triple $\tau$, that is:
\begin{eqnarray}\label{rec5}H'_k=\mathcal{F}_k\cup\{\tau\}\cup 01\mathcal{E}_{k-1}\cup 01\mathcal{F}_{k-1}.\end{eqnarray}\end{proof}

\begin{example} Example~\ref{ex6} uses $\mathcal{T}_3$ in display~(\ref{rec2}), with $S_1(\epsilon)=012$ and $S_3=034$ yielding the hypergraph $\mathcal{T}_3$ depicted in the lower left of Fig.~\ref{fig6}. Example~\ref{ex7} uses $H'_k$ in display~(\ref{rec5}) for $k=4$, $\mathcal{F}_4$ and $\mathcal{E}_3$ in display~(\ref{rec2}) and $\tau$ in display~(\ref{rec4}), with $S_1(01)=67a$, $S_2=875$, $0\underline{S}_31=02a$, $S_1(\epsilon)=146$, being these four triples the elements in $\mathcal{F}_4$; $01S_4=3bcd$, this one as the only element of $01\mathcal{E}_3$, (while $\mathcal{F}_3=\emptyset$); and $\tau=02a$, yielding the hypergraph $H'_4$ depicted at the lower left corner of Fig.~\ref{fig7}.\end{example}

\begin{corollary}\label{the-end}  To each Hamilton cycle in $O_k$ produced by Theorem~\ref{L6} corresponds a Hamilton cycle in $M_k$.
\end{corollary}

\begin{proof}
For each vertical list $L(\alpha)$ provided by Theorem~\ref{L5},  let $L^M(\alpha)$ be a vertical list as exemplified in Example~\ref{rrr2} and Fig.~\ref{fig3}, which is obtained from $L(\alpha)$ by replacing its ``$=$" signs by: {\bf(a)} ``$>$" signs (meaning left-to-right string-reading) for the strings $L_{2j}(\alpha)$ ($j\in[0,k]$) of $L(\alpha)$ and {\bf(b)} ``$<$" signs (meaning right-to-left string-reading) for the strings $L_{2j+1}(\alpha)$ ($j\in[0,k-1]$) of $L(\alpha)$.  Then, Theorem~\ref{L6} can be adapted to producing Hamilton cycles in the $M_k$ by repeating the argument in its proof in replacing the lists $L(\alpha)$ by lists $L^M(\alpha)$, since they have locally similar behavior, being the cycles provided by the lists $L^M(\alpha)$ twice as long as the corresponding lists $L(\alpha)$, so the said local behavior happens twice around opposite (rather short) subpaths. Combining Dyck-word triples and quadruples as in display~(\ref{!}) into adequate pullback liftings (of the covering graph map $M_k\rightarrow O_k$ associated to item (ii), Section~\ref{s1}) in the lists $L^M(\alpha)$ of those parts of the lists $L(\alpha)$ in which the necessary symmetric differences take place to produce the Hamilton cycles in $O_k$ will produce corresponding Hamilton cycles in $M_k$.
\end{proof}

\noindent{\bf Historical Note.} The $k$-edge ordered trees appearing in \cite[p.~221, item~(e)]{Stanley} as ``plane trees with" $k+1$ vertices and in \cite{gmn} as ``ordered rooted trees", represent Dyck paths of length $2k$ (see Subsection~\ref{Dyck}). These trees are equivalent to $k$-strings $0b_{k-1}\cdots b_1$ called $k$-{\it RGS}'s in \cite{D2} and tailored from the RGS's of Section~\ref{germs} via items (r) and (u) in \cite[p.~222]{Stanley} in a different way from that of the $k$-germs of Section~\ref{germs}. An equivalence of these $k$-germs and those $k$-RGS's was presented in \cite{D2} via their distinct relation to the $k$-edge ordered trees, whose purpose in \cite{gmn,M} was using their plane rotations toward Hamilton cycles in $M_k$, not related to the odd-graph approach to Hamilton cycles of \cite{Hcs} to which we applied our ideas in Section~\ref{verL}.








\begin{thebibliography}{99}

\bibitem{Arndt} J. Arndt, Matters Computational: Ideas, Algorithms, Source Code, Springer, 2011.

\bibitem{B} N. Biggs, Norman, {\it Some odd graph theory}, Annals of the New York Academy of Sciences, {\bf 319} (1979), 71--81.

\bibitem{DD} I. J. Dejter, {\it On coloring the arcs of biregular graphs}, Applied Discrete Math., {\bf 284} (2018), 489--498.

\bibitem{D2} I. J. Dejter, {\it A numeral system for the middle-levels graphs}, Electronic Journal of Graph Theory and Applications, {\bf 9} (2019), 137-156.

\bibitem{D1} I. J. Dejter, {\it Reinterpreting the middle-levels theorem via natural enumeration of ordered trees}, Open Journal of Discrete Applied Mathematics, {\bf 3} (2020), 8--22.

\bibitem {GR} C. Godsil and G. Royle, Algebraic Graph Theory, Springer 2011.

\bibitem{DKS} D. A. Duffus, H. A. Kierstead and H. S. Snevily, {\it An explicit 1-factorization in the
middle of the Boolean lattice}, Journal of Combinatorial Theory, Ser. A, {\bf 65} (1994), 334--342.

\bibitem{KT} H.\ A.\ Kierstead and W.\ T.\ Trotter, {\it Explicit matchings in the middle levels of the Boolean lattice}, Order, {\bf 5} (1988), 163--171.

\bibitem{gmn} P. Gregor, T. M\"utze and J. Nummenpalo, {\it A short proof of the middle levels theorem}, Discrete Analysis, 2018:8, 12pp.

\bibitem{M} T. M\"{u}tze, {\it Proof of the middle levels conjecture}, Proceedings of the London Mathematical Society, {\bf112} (2016) 677--713.

\bibitem{u2f}
T. M\"utze, C. Standke, and V. Wiechert,  {\it A minimum-change version of the Chung--Feller theorem for
Dyck paths}, European J. Combin., {\bf 69} (2018), 260--275.

\bibitem{Hcs} T. M\"utze, J. Nummenpalo and B. Walczak, {\it Sparse Kneser graphs are hamiltonian}, Journal of the London Mathematical Society,  {\bf 103} (2021),
1253--1275.

\bibitem{oeis} N.\ J.\ A.\ Sloane, The On-Line Encyclopedia of Integer Sequences, \url{http://oeis.org/}.

\bibitem{Stanley} R. Stanley, Enumerative Combinatorics, Volume 2, (Cambridge Studies in Advanced Mathematics Book 62), Cambridge University Press, 1999.
\end{thebibliography}
\end{document}